\documentclass{amsart}[14pt]
\usepackage{amssymb, amsfonts}
\usepackage{hyperref}

\newtheorem{theorem}{Theorem}[section]
\newtheorem{lemma}[theorem]{Lemma}
\newtheorem{proposition}[theorem]{Proposition}
\newtheorem{corollary}[theorem]{Corollary}

\theoremstyle{definition}
\newtheorem{definition}[theorem]{Definition}
\newtheorem{example}[theorem]{Example}

\newtheorem{question}[theorem]{Question}

\theoremstyle{remark}

\newcommand{\cA}{{\mathcal A}}
\newcommand{\cB}{{\mathcal B}}
\newcommand{\cC}{{\mathcal C}}
\newcommand{\cD}{{\mathcal D}}

\newcommand{\cH}{{\mathcal H}}

\newcommand{\cK}{{\mathcal K}}

\newcommand{\cM}{{\mathcal M}}

\newcommand{\cS}{{\mathcal S}}

\newcommand{\bC}{{\mathbb C}}
\newcommand{\bN}{{\mathbb N}}
\newcommand{\bR}{{\mathbb R}}

\newcommand{\bZ}{{\mathbb Z}}

\renewcommand{\a}{\alpha}
\renewcommand{\b}{\beta}
\renewcommand{\d}{\delta}
\newcommand{\D}{\Delta}
\newcommand{\e}{\epsilon}
\newcommand{\g}{\gamma}
\newcommand{\G}{\Gamma}

\newcommand{\om}{\omega}

\newcommand{\s}{\sigma}

\newcommand{\<}{\langle}
\renewcommand{\>}{\rangle}

\newcommand{\hr}{{\mathsf{h}}}

\newcommand{\ep}{{\varepsilon}}

\newcommand{\af}{\cA_c}
\newcommand{\hd}{\widehat \D}

\newcommand{\Lat}{{Latr\'emoli\`ere}}

\numberwithin{equation}{section}

\begin{document}

\title[Convergence of Fourier truncations ]
{Convergence of Fourier truncations  \\   for compact quantum   
groups  \\
and finitely generated groups}

% Remove or comment out any unused author tags.
% author one information
\author{Marc A. Rieffel}
\address{Department of Mathematics\\
University of California\\
Berkeley, CA\ \ 94720-3840}
%\curraddr{}
\email{rieffel@math.berkeley.edu}
%\thanks{The research reported here was supported in part by}
%\dedicatory{In celebration of the }
   
%\subjclassname}{
\subjclass[2010]{
Primary 58B32; Secondary 46L87,  58B34}
% Use this one if  are using an older version of amsproc.
%\subjclass{}

\keywords{truncation, compact quantum group, finitely
generated group, quantum metric , 
quantum Gromov-Hausdorff distance, operator system}
%\date{}

\large{

\begin{abstract}
We generalize the Fej\'er-Riesz operator systems defined for the circle group by Connes and van Suijlekom to the setting of compact matrix quantum groups and their ergodic actions on C*-algebras. 
These truncations form filtrations of the containing C*-algebra. We show that when they and the containing C*-algebra are equipped with suitable quantum metrics, then under suitable conditions they converge to the
containing C*-algebra for quantum Gromov-Hausdorff distance.
Among other examples, our results are applicable to the
quantum groups $SU_q(2)$ and their homogeneous
spaces $S^2_q$.
\end{abstract}

\maketitle

\tableofcontents

%%%%%%%%%%%%%%%%%%%

\section{Introduction}

Over the past decade or so there has been slowly increasing 
interest in the use of operator systems in non-commutative 
geometry 
\cite{Krr, KrL, Sai, GlS20, GlS, DnLL, CvS2, AgKK, AgKK2, ciF, KKd, dng, LvS, DRH, vNl22}, 
spurred on quite recently by the paper \cite{CnvS}
of Connes and van Suijlekom
concerning spectral truncations. Even more recently, 
the use of quantum Gromov-Hausdorff distance
for spectral truncations in the setting of operator systems was 
initiated in \cite{vSj2}. (But see the earlier paper \cite{DnLM}
which used the setting of order-unit spaces.)
These two papers
of Connes and van Suijlekom concentrate 
on revealing the remarkably rich structure that one obtains 
already in the case of the circle. (See also \cite{Hkk, LvS}.)

For a spectral triple $(\cA,\cH, D)$ the spectral truncations
are the compressions $ P\cA P $, where $ P $ is a 
finite-dimensional spectral projection of $D $, viewed as
an operator system of operators on $ P\cH $. In \cite{vSj2}
it was shown that, for the case in which $\cA $
is the algebra $C (S^1) $ of continuous functions on the circle,
the spectral truncations converge to  $C (S^1) $ for quantum
Gromov-Hausdorff distance. (Very recently the corresponding question
for higher-dimensional tori has been explored in \cite{LvS}.) 
It is extremely interesting to ask
to what extent this continues to be true for other spectral triples.

In \cite{CnvS, vSj2} the operator systems 
dual to the spectral truncations
play an important role. While for the circle the spectral truncations 
are shown to be the vector spaces of Toeplitz matrices of various 
sizes, their operator system duals are 
shown to be the vector spaces
of functions whose Fourier coefficients 
are 0 outside given intervals
of $\bZ $ symmetric about 0. In \cite{vSj2} it is shown that these
dual operator systems too converge to $C (S^1) $ for a suitable 
quantum Gromov-Hausdorff distance. It is interesting to ask
to what extent this continues to be true in other situations.

The purpose of the present paper is to give an answer to this latter
question in a quite broad setting, namely that of compact matrix 
quantum groups, which includes the case 
of finitely generated discrete groups. Our setting 
also includes ergodic actions of compact
matrix quantum groups on unital C*-algebras. It seems to me
that the feature underlying the proof in \cite{vSj2} 
of the convergence
of the dual operator systems is the fact 
that $S^1 $ is a compact group.

We now give an imprecise statement of our main theorem. 
The various technical terms will be defined as needed in the 
text of this paper.

\begin{theorem} (Imprecise statement of Theorem \ref{main})
Let $ (\cA, \D) $ be a compact matrix quantum group, and let 
$\a $ be an ergodic action of $ (\cA, \D) $ on a unital C*-algebra $\cB $.
Let $\cA_1$ be the operator subsystem of $\cA $ generated by
the coordinate elements of a fixed faithful finite dimensional 
corepresentation of $\cA $, and for each $n \in \bZ $ let $\cA_n $
be the operator subsystem of $\cA $ generated by all products of
$n $ elements of $\cA_1$, so that the $\cA_n $'s form a filtration
of $\cA $. Let $\{\cB_n\}$ be the corresponding filtration of $\cB $.
Assume that $ (\cA, \D) $ is coamenable, and let $L^\cB$
be a regular $\a$-invariant Lip-norm on $ \cB $. For each $n $
let $L^\cB_n$ be the restriction of $ L^\cB $ to $\cB_n $.
Then the compact quantum metric spaces $ (\cB_n, L^\cB_n)$
converge to $ (\cB, L^\cB)$ for quantum Gromov-Hausdorff
distance.
\end{theorem}

We will call the $\cB_n $'s ``Fourier truncations'' of
 $\cB $, for reasons that will be evident later.
They correspond to the Fej\'er-Riesz operator systems
of \cite{CnvS, vSj2}. So we are dealing with
``truncations of the Fourier modes of the (bosonic) elements'' 
as mentioned very briefly in section 6 of \cite{CnvS}.

Of course there remains the big challenge of 
formulating and proving corresponding results for 
the spectral truncations of \cite{CnvS, vSj2}.

In Section \ref{secex} we provide a number of classes of
examples to which our results are
applicable. These include
ordinary compact Lie groups and their ergodic
actions on unital C*-algebras, the group C*-algebras of 
amenable finitely generated discrete groups, groups and
quantum groups of rapid decay, 
and the quantum groups
$SU_q(2)$ and their homogeneous
space $S^2_q$ \cite{AgK, AgKK, KKd}.

The proof of our main theorem makes crucial use of 
parts of Hanfeng Li's exploration of metric aspects of
ergodic actions of compact quantum groups in 
section 8 of \cite{Lih5}.

I thank David Kyed and Hanfeng Li for useful comments
on the first version of this paper.

%%%%%%%%%%%%%%%%%%%%%%%%%%%%%%

\section{Preliminaries about compact quantum groups}
\label{secprim}

In this section we gather well-known facts about compact quantum 
groups that we will need. We give few proofs, 
and instead refer the reader to \cite{Wr98, DCm, vD98, DijK}
and the references they include.

A compact quantum group is a pair $(\cA, \D)$, where $\cA$ is a 
unital C*-algebra and $\D: \cA \to  \cA \otimes \cA$ is a 
unital $*$-homomorphism, called the comultiplication, such that
\begin{enumerate}
\item $(\D \otimes I^\cA)\D = (I^\cA \otimes \D)\D$ as 
homomorphisms $\cA \to \cA \otimes \cA \otimes \cA$ (coassociativity) 
\item The spaces $(\cA \otimes 1_\cA)\D(\cA) := 
\mathrm{span}\{(a \otimes 1_\cA)\D(b) : a, b \in \cA\}$ 
and $(1_\cA \otimes \cA)\D(\cA)$ are
dense in $\cA \otimes \cA$ (the cancellation property).
\end{enumerate}
\noindent
Here $I^\cA$ is the identity operator on $\cA$, and the tensor 
product is the minimal tensor product for C*-algebras.

A fundamental fact about compact quantum groups is that one 
can prove that they have a generalization of the normalized 
Haar measure on ordinary compact groups. That is, there is 
a unique state, the ``Haar state'' $\hr :\cA \to \bC$ that 
satisfies the left and right invariance conditions
\[
(I^\cA \otimes \hr) \circ \D = \hr 1_\cA   = (\hr \otimes I^\cA) \circ \D    .
\]
In general the Haar state may not be a trace, and it may not be 
faithful on $\cA$. 

The following class of examples is of importance for this paper.

\begin{example}
\label{dg}
Let $\G$ be a discrete group. Denote by $\cC^*_u(\G)$ the 
full group C*-algebra of $\G$. Because $\G$ is discrete, 
this is a unital C*-algebra. Every $x \in \G$ determines
a unitary element, $\d_x$, of $\cC^*_u(\G)$, and the linear span, 
$C_f(\G)$, of these 
unitary elements forms a dense $*$-subalgebra of $\cC^*_u(\G)$. 
This subalgebra
can be viewed as consisting of the finitely supported functions on $\G$.
The mapping $\d_x \mapsto \d_x \otimes \d_x$ determines a unital homomorphism, 
$\D$, from $\cC^*_u(\G)$  into $\cC^*_u(\G) \otimes \cC^*_u(\G)$
that makes $\cC^*_u(\G)$ into a compact quantum group. 
(Note that in this situation one might expect to use 
the maximal C*-tensor product, but for quantum groups
it is usual to use the minimal C*-tensor product, 
for which the construction of this example remains valid.) 
For this class of examples the Haar state is just the usual tracial state 
coming from evaluating functions in $C_f(\G)$
at the identity element of $\G$. 
(For the details of this construction see \cite{BdMT}; for even  locally compact $\G$
see the earlier \cite{Ior}.) In the case that $\G$ is not amenable, 
the Haar state on $\cC^*_u(G)$  is not faithful, 
almost by one definition of amenablilty.
(But the Haar state is always faithful on $C_f(\G)$.)
The reduced C*-algebra of $\G$, 
defined by the GNS
construction for the Haar state,
is a compact quantum group in essentially the same way.
If $\G=F_2 $, the free group on two generators, then the 
Haar state is faithful on its reduced C*-algebra, but $F_2$
is not amenable.
\end{example}

For a compact quantum group $(\cA, \D)$ the 
C*-algebra $\cA$ gives the ``space'',
and has its own $*$-representations on Hilbert spaces. The generalization 
 of  finite-dimensional unitary representations of groups 
 to the quantum group context is
called  finite-dimensional unitary  ``corepresentations''. 
Let $\cH$ be a finite dimensional Hilbert space 
(with inner product linear in the second variable). 
Consider the right Hilbert $\cA$-module $\cH \otimes \cA$, with 
$\cA $-valued inner product
$\<\xi \otimes a,\eta \otimes b\>_\cA =\<\xi, \eta\>a^*b$ for $a,b\in \cA$ and
$\xi,\eta \in \cH$.
A unitary corepresentation of $(\cA, \D)$ on $\cH$ is a linear map
$u: \cH \to \cH \otimes \cA$ that satisfies
\begin{enumerate}
\item $(u \otimes I^\cA)\circ u = (I^\cH \otimes \D) \circ u$
\item $\<u(\xi),u(\eta)\>_\cA = \<\xi ,\eta\>\bC 1_\cA$ for $\xi, \eta \in \cH$
\item $u(\cH)\cA = \cH \otimes \cA$.
\end{enumerate}
(See \cite{DCm}, section 3 of \cite{Pn07} and definition 2.8 of \cite{Sai}.)
Given $\xi, \eta \in \cH$, we obtain an 
element, $u_{\xi\eta}$ of $\cA$, by
\[
u_{\xi\eta} = (\<\xi, \cdot|  \otimes I^\cA)(u(\eta)) 
\]
(where $\<\xi, \cdot| $ denotes the linear functional on $\cH$
determined by $\xi$).
These elements are called ``coordinate elements" of the corepresentation
$u $. If $\xi_1, \cdots, \xi_n $ is an orthonormal basis for $\cH $, and if we set
$u_{jk} = u_{\xi_j\xi_k} $, then the matrix $U = \{u_{jk}\} $ is a unitary
element of $M_n(\cA) $ that satisfies
\begin{equation}
\label{unmat}
\D(u_{jk}) = \sum ^n_{\ell =1} u_{j\ell} \otimes u_{\ell k} . 
\end{equation}
Such a matrix provides the more common definition 
of a finite-dimensional unitary corepresentation of a 
compact quantum group. Given such a unitary matrix $ U $ that satisfies this
relation, one can define the corresponding linear map 
$u: \cH \to \cH \otimes \cA$ by $u (\xi) = U(\xi \otimes 1_\cA)$, 
suitably interpreted.
See lemma 1.7 of \cite{DCm}.

We will let $\hd $ denote the set of unitary equivalence classes of
(finite-dimensional) irreducible 
 unitary corepresentations of $ (\cA, \D) $. It can be considered
 to be the set of Fourier modes in a generalized sense.
 In the usual 
 way, our notation will not distinguish between such 
corepresentations
and their equivalence classes. For each $\g\in\hd $ we let 
$\cA^\g $ denote the corresponding isotypic component of $\cA $.
It is a finite-dimensional subspace of $\cA $, and 
$\D (\cA^\g)\subseteq \cA^\g \otimes \cA^\g $. Each unitary 
corepresentation has a conjugate unitary corepresentation. 
We denote the conjugate of $\g $ by $\g^c $.
Then $ (\cA^\g) ^* =\cA^ {\g^c} $. The isotypic 
components are mutually orthogonal with respect 
to the inner product defined by the Haar state.

Let $\af$ be the linear span of all the coordinate 
elements of all finite dimensional 
corepresentations of $(\cA, \D)$, or equivalently, 
of all the matrix elements 
$u_{jk} $ for representatives $U$ of all the unitary equivalence 
classes of irreducible corepresentations. It is the
algebraic direct sum of the  $\cA^\g $'s, and
it is a dense subspace of $\cA$. Because one can 
form tensor products of corepresentations, 
$\af $ is a unital subalgebra 
of $\cA $. Because one can form conjugates 
of corepresentations, $\af $ is a    
$* $-subalgebra of $\cA $. Furthermore,  the one-dimensional 
trivial corepresentation is well-defined on $\af $ and serves as a coidentity 
element for $\af $, and there is a well-defined 
coinverse (i.e. antipode)
on $\af $. Thus $\af $ is a unital Hopf $*$-algebra \cite{DCm}. 
In addition, 
the Haar state on $\cA $ is faithful on $\af $. 

Let $\cA_u$ be the completion  of $\af $ for the universal C*-norm on $\af $. 
All of the unital Hopf $* $-algebra structure of $\af $ extends 
to $\cA_u$ 
except that the coinverse may not be continuous for the universal 
C*-norm on $\af $. If it is continuous, then $\af $ is said to be of
``Kac type''. The Haar state need not be faithful on $\cA_u $, 
as seen in Example \ref{dg}.

Instead, let $\cA_r$ be the completion of 
$\af $ for the C*-norm from the
GNS representation on the Hilbert space $ L^2 (\cA_c , \hr) $
obtained from using the Haar 
state. All of the unital Hopf $* $-algebra structure of $\af $ 
extends to $\cA_r$ 
except that the coinverse and the coidentity 
may not be continuous for the reduced
norm. We call $\cA_r$ the ``reduced'' compact quantum group 
for $\af $.  The Haar state is faithful on $\cA_r$.     
 The coidentity is continuous for the reduced norm 
 if and only if $\cA_u$ and
 $\cA_r$ coincide, in which case $\cA $ is said to be
``coamenable''.
For a discrete group $ \G $ the corresponding 
compact quantum 
group is coamenable exactly if $ \G $ is amenable.
In this paper, we will be concerned mostly
with compact quantum groups that are coamenable.

Let $\cA'$ be the Banach space of all continuous linear 
functionals on $\cA $. For $ \mu, \nu \in \cA'$ set
$\mu * \nu = (\mu \otimes \nu)\circ \D$. This is an associative 
product, ``convolution'', on $\cA' $ for 
which $\cA' $ is a Banach algebra.
The Haar state $\hr $ is an element of $\cA'$, and
$\hr * \mu = \mu(1_\cA)\hr = \mu * \hr$ for all $\mu \in \cA'$. 
If the coidentity element is continuous on $\cA$ so that it is
an element of $\cA'$, then it serves as an identity element
of the algebra $\cA'$, and we will then denote it by $\ep $.

For any $\b, \g \in \hd$ let $\b \otimes \g$ denote the subset 
of $\hd$ consisting of all the elements of $\hd$ that appear in the
decomposition of the tensor product of $\b$ with $\g$ into
irreducible representations. For any finite subset $S$ of $\hd$
let $\cA^S$ denote the direct sum of the $\cA^\g$'s with
$\g \in S$. Then for any $\b, \g \in \hd $ one has 
$\cA^\b \cA^\g \subseteq \cA^{\b \otimes \g}$ (where
$\cA^\b \cA^\g$ means ``finite sums of products'').

The following result will be crucial for our purposes.

\begin{proposition}
\label{provs}
Let $a\in\cA_c $, viewed as an element of  $ L^2 (\cA_c , \hr) $, 
and satisfying $\|a\|_2 = 1 $ for the norm on $ L^2 (\cA_c , \hr) $.
Let $\mu_a $ be the vector state on $\cA_r $ determined by
$a $. Then there is a finite subset, $ F_a $, of $\hd $
such that if $\g\in\hd $ 
but $\g $ is \emph{not} in $ F_a $, 
then $\mu_a (\cA^\g) = 0 $.
\end{proposition}
 
\begin{proof}
Since $a\in\cA_c $, it is the finite sum of non-zero elements 
of various $\cA^\g $'s. Let $ F $ be the set of the corresponding
$\g $'s. Let $\d\in\hd $ and let $b\in\cA^\d $. Suppose that
$\mu_a (b)\neq 0 $, that is, $\hr (a^*ba)\neq 0 $. Then there 
must be $\g,\b \in F $, and $a_\g\in \cA^\g $ and 
$a_\b \in\cA^\b $, such that $\hr (a_\g^*ba_\b)\neq 0 $.

Now 
$\hr $ may not be a trace. But in proposition 3.12 
of \cite{vD98} Van Daele shows that there is an automorphism,
$\s $ of $\cA_c $ with the property that 
$\hr (ab) = \hr (b\s (a)) $ for all $a, b \in \cA_c $, and also that
$\hr (\s (a)) =\hr (a) $.  Van Daele calls this a ``weak KMS 
property''. (To see the relationship with KMS automorphisms 
see theorem 1.4 of \cite{Wr98} or section 2.1 of \cite{Tm08}.)
However, $\s $ is in general not a
$* $-automorphism. But $\s $ does 
carry isotypic components of $\cA_c $ into themselves.
To see this, let $\g \in \hd$ and let $d\in\cA^\g $ be given. Then for any 
$\b \in \hd$ with $\b\neq\g $, and any $c\in\cA^\b $ we have
$ 0 =\<c^*, d^*\> =\hr (dc^*) =\hr (c^*\s (d)) =\<\s (d),c\> $.
This implies that $\s (d)\in\cA^\g $. 

Returning to our original 
$\g $ and $\b $, we see that 
$\hr (a_\g^*ba_\b) =\hr (ba_\b\s (a_\g^*)) $, 
and that $\s (a_\g^*) $ is 
in $\cA^{\g^c} $. Since
$0  \neq \hr (ba_\b\s (a_\g^*)) = \<a_\b\s (a_\g^*), b^*\>$ 
and $a_\b\s (a_\g^*)$ is in $\cA^{\b \otimes \g^c}$
it follows that $\d $ is in $ \g\otimes \b ^c $.  

For $ F$ as defined above, let $ F^c $ denote the set of
all conjugates of elements of $ F $. Then let $ F_a $ 
be the set of all elements of $\hd $ that are contained in the 
tensor product of some element of $ F $ with some 
element of $ F^c $, in that order. We see that if $\g\in\hd $ 
and if $\g $ is \emph{not} in $ F_a $, 
then $\mu_a (\cA^\g) = 0 $. Note that the 
set $ F_a $ is finite.
\end{proof} 
 
Essentially by definition, a compact Lie group (not necessarily connected) 
has a faithful finite-dimensional unitary representation. 
By taking the direct sum of this representation with its 
conjugate if necessary, 
we can assume that this representation is self-conjugate.
By then taking the direct sum with the trivial one-dimensional representation if necessary, we can assume that this representation 
is self-conjugate and contains the trivial representation. 
There are many choices of such a faithful representation.
Let $ G $ be a compact Lie group, and let $(\cH, U)$ be a faithful 
finite-dimensional unitary representation that 
is self-conjugate and contains the trivial representation. 
Let $C(G;U)$ be the linear span of the coordinate functions of $ U $.
Then $ C (G; U) $ is a finite-dimensional subspace of $ C (G) $ that is closed under complex conjugation and contains the constant functions.
Thus the set of $\bR $-valued functions that it contains 
is an order-unit space, and when $ C (G) $ is viewed as a
commutative C*-algebra, 
we see that $ C (G; U) $ is an operator system.
Because the representation $(\cH, U)$ is faithful, the functions 
in $C (G; U) $ separate the points of $G $. Thus the 
(algebraic) subalgebra of $ C (G) $ generated by $ C (G; U) $ is dense 
in $C (G) $, according to the Stone-Weierstrass theorem. 
(Because $C (G; U) $ is finite-dimensional, we say that $ C (G) $
is ``finitely generated'', with $C (G; U) $ as generating set.) 
Let $\cA=C (G) $, and for 
 $n \in \bN $ let $\cA_n $  
be the linear span of products of $n $ elements of $C (G; U) $
(with $\cA_0 = \bC 1_\cA $).
Then each $\cA_n $ is a finite-dimensional operator system
in $\cA $, and for any $m, n \in \bN $ we have
$\cA_m \cA_n \subseteq \cA_{m+n} $. Thus
$\{\cA_n\} $ is a filtration of the C*-algebra $\cA$.
We view the $\cA_n $'s as Fourier truncations of $\cA$.
We now generalize this structure to the setting of compact 
quantum groups.

Let $(\cA, \D)$ be a compact quantum group and let 
$(\cH, u) $ be a finite-dimensional corepresentation of
$(\cA, \D)$. Much as above, we can arrange that this
corepresentation is self-conjugate and contains the 
trivial corepresentation (which 
is $u(z) = z \otimes 1_\cA$ for $z \in \bC$). 
Let $\cA_1 $ be the linear span of the coordinate 
elements of $u $, as defined above. It is a finite-dimensional
operator system in $\cA $. For each $n\in\bN $ 
let $\cA_n $  
be the linear span of products of $n $ elements of $\cA_1 $
(with $\cA_0 = \bC 1_\cA $).
Then each $\cA_n $ is a finite-dimensional operator system
in $\cA $, and for any $m, n \in \bN $ we have
$\cA_m \cA_n \subseteq \cA_{m+n} $. But in general
the union of the $\cA_n $'s is not dense in $\cA $.

\begin{definition}
\label{deffnd}
With notation as just above, we say that a finite-dimensional corepresentation
$(\cH, u) $ is \emph{faithful} if the union of all the $\cA_n $'s is  
dense in $\cA $ (so $\cA $ is finitely generated and 
the $\cA_n$'s form a filtration of $\cA$). A faithful 
corepresentation is often called a ``fundamental'' 
corepresentation. In the context of
\cite{CnvS, vSj2} we view the $\cA_n$'s as Fourier truncations
of $\cA$.

A compact quantum group that has a faithful finite-dimensional 
corepresentation is called a ``compact matrix quantum group''.
These are exactly the quantum generalization of compact Lie
groups.
\end{definition}

\begin{example}
\label{fingen}
Let $ \G $ be a finitely generated group. We can choose a finite 
generating set $S $ that is closed under taking inverses
and contains the identity element of $\G $. The irreducible
corepresentations of $\cA = C^*(\G  )$ (all versions) 
are all one-dimensional,
and correspond to the elements of $\G $. Given $x \in\G $,
for the corresponding corepresentation $u^x $
we can choose $\cH =\bC $, and define $u^x $ by
$u^x (z) = z\d_x \in \cH \otimes \cA = \cA$ for $z\in\bC $.
Let $u^S$ be the direct sum of the $u^x $'s as $x $
ranges over $S$. The range of $u^S $ can be naturally 
identified with  the subspace of
functions supported on $S $. It forms a finite dimensional 
operator system in $\cA $ that generates $\cA $. Thus it
determines a familiar filtration of $\cA =C^*(\G)$
(used e.g. in \cite{R18, OzR, ChRi}), whose elements
we view as Fourier truncations of $C^*(\G)$. 
In the case that $\G$ is $\bZ$ the elements of 
this filtration are exactly the
Fej\'er-Riesz operator systems of \cite{CnvS, vSj2}. 
\end{example}

%%%%%%%%%%%%%%%%%%%%%%%%%%%

\section{Preliminaries about actions}
\label{secact}

Just as actions of compact groups on compact spaces are 
of much importance, so too, actions of compact quantum 
groups on unital C*-algebras are of much importance.
We recall here the facts that we will need. For details, see
for example
\cite{BC05, Pd95, BdMT, DCm, Slt}.
Let $ (\cA,\D) $ be a compact quantum group, and let $\cB $
be a unital C*-algebra. By a (left) action of $ (\cA,\D) $ on 
$\cB $ we mean a unital injective 
$*$-homomorphism, $\a $, from $\cB $
into $\cB \otimes \cA $ such that
\begin{equation}
\label{eqact}
(\a \otimes I^\cA) \circ \a = (I^\cB \otimes \D)\circ \a
\end{equation}
and the linear span of $\a(\cB)(1_\cB \otimes \cA) $ 
is dense in $\cB \otimes \cA $. (Often $\a$ is called 
a ``coaction''. Also, some authors call this a ``right'' action. 
But Li in \cite{Lih5} calls it a ``left'' action, and since we make
important use of his results there we will follow his usage.) 
Notice that $\D $ can be viewed as 
giving an action of $\cA $ on itself that is appropriately viewed
as the ``left regular action'' of this quantum group.

 For each $\g \in \hd$ let
 $U^\g = \{u^\g_{jk}\} $ be a unitary matrix
in $M_n(\cA) $ that represents $\g $, and thus satisfies
\[
\D(u^\g_{jk}) = \sum ^n_{\ell =1} u^\g_{j\ell} \otimes u^\g_{\ell k} . 
\]
Then, as described in section 2 of \cite{Slt} (and a number
of other places beginning with \cite{Pd95} ),
there are elements $ \phi^\g_{ij}$ of $\cA'$ such that
\[
   \phi^\b_{ij}(u^\g_{kl}) = \d_{\b\g}\d_{ik}\d_{jl}
\]
for all $\b, \g$ and all appropriate {i, j, k, l}. 
For each $\g \in \hd $ define an operator, $ E^\g_\a $,
from $\cB $ to $\cB \otimes \bC = \cB$ by
\[
E_\a^\g = \sum_{k=1}^{n_\g} (I^\cB \otimes \phi^\g_{kk}) \circ \a  ,
\]
where $n_\g$ is the dimension of the corepresentation $\g $.
(Notice the hint of a trace in this formula.)
Then $ E^\g_\a $ is the projection onto the $\g $-isotypic 
component,  $\cB^\g $, of
$\cB $ for the action $\a $. The product of any two of these projections 
for different $\g $'s is 0. Also, 
$\a (\cB^\g)\subseteq \cB^\g\otimes \cA^\g $.

We denote the isotypic component of the trivial corepresentation 
by $\cB^\a $. It consists
exactly of the elements $b $ of $\cB $ that are $\a $-invariant,
that is, $\cB^\a = \{b \in \cB: \a(b) = b\otimes 1_\cA\}$. In 
lemma 4 of 
\cite{BC05} Boca shows that there is a canonical conditional 
expectation, $ E$, from $\cB $ onto $\cB^\a $, given by
$ E (b) = (I^\cB \otimes \hr)\a (b) $.

For our purposes we want the action $\a $ to be ergodic, that is, 
the isotypic component of the trivial representation, $\cB^\a $ 
should be 
exactly $\bC 1_\cB $. In this case, the conditional expectation 
$E$ is of the form $ E (b) = \om (b)1_\cB $ where $\om $ is the 
unique $\a $-invariant state on $\cB $, where $\phi  \in S (\cB) $ 
is said to be $\a $-invariant if 
$ (\phi \otimes \mu)\a (b) = \mu (1_ \cA)\phi (b) $ for all 
$\mu\in\cA'$ and $b\in \cB $. (Ergodic actions are sometimes
called ``quantum homogeneous spaces'', e.g. \cite{hsh}.)

Boca proved \cite{BC05}
that if $\a $ is ergodic then all of the isotypic components $\cB^\g $
are finite dimensional. Let $\cB_c $ be the algebraic direct sum
of all the isotypic components. It is a dense $*$-subalgebra
of $\cB$ that is an analog of the coordinate subalgebra
of $\cA$, and it is sometimes called the subalgebra
of ``regular'' elements of $\cB $.

For a finite subset $S$ of $\hd $
let us set $\cB^S = \bigoplus_{\g \in S} \cB^\g$.
Suppose that $ S $ is closed under taking conjugate 
corepresentations
and contains the trivial corepresentation. 
Then $\cB^S $ is an operator system in $\cB $.
For each 
$n \in\bN $ let $S^n$ denote the collection of all irreducible 
corepresentations that are contained in the $n$-fold
tensor products of elements of $S $. Let $S^0$ consist of just 
the trivial corepresentation, Then the $ S^n $'s form an increasing 
sequence of finite subsets of $ \hd$, each of which is closed under 
taking conjugate corepresentations,
and contains the trivial corepresentation. Thus the $\cB^{S^n}$'s
form an increasing sequence of operator systems in $\cB$.
Suppose further that $(\cA, \D)$ is a compact matrix quantum group,
and that the direct sum of the irreducible corepresentations in 
$ S $ is a faithful corepresentation of $(\cA, \D)$. 
Then the $\cB^{S^n}$'s form a filtration of $\cB$. These are 
the filtrations in which we are interested. In the context of
\cite{CnvS, vSj2} we can view the $\cB_n$'s as 
Fourier truncations of $\cB$.

We summarize the above discussion with:

\begin{proposition}
\label{profil}
Let $(\cA, \D)$ be a compact matrix quantum group,
and let $\a $ be an ergodic action of $(\cA, \D)$
on a unital C*-algebra $\cB $. Let $S $ be a finite subset 
of $\hd $ such that the direct sum of the elements of $S$ 
is a faithful unitary corepresentation of $(\cA, \D)$, and that
$S $ is closed under taking conjugate corepresentations,
and contains the trivial corepresentation. Then, with notation
as above, the Fourier truncations $\cB^{S^n}$'s 
form a filtration of $\cB$. 
\end{proposition}

%%%%%%%%%%%%%%%%%%%%%%

\section{Quantum metrics}
\label{secmet}
The definition of quantum metrics was given in \cite{R5} in the setting of 
order-unit spaces. Here we only need the definition for the case of 
operator systems (notably those of the filtrations 
discussed above), so we now recall the definition for that 
case \cite{Krr, KrL, Lih5, DRH}. 
We will need the fact that an operator system $\cC $ has 
a well-defined state space, which we will denote by $ S (\cC) $. 
It is compact for the weak-$* $ topology.
Let $\cC $  be an operator system. A quantum metric on $\cC $ is 
a seminorm, $L $, on $\cC $ that plays the role of assigning the Lipschitz
constant to functions on a compact metric space. 
As such, it can take the value 
$+\infty $. 
\begin{definition} 
\label{deflip}
Let $\cC $ be an operator system, and let $ L $ be a seminorm on $\cC $ 
that may take the value $+\infty $. We say that
$ L$ is a \emph{Lip-norm} if it satisfies the following properties: 
\begin{enumerate}
\item For any $c\in \cC $ we have $ L (c^*) = L (c) $, and 
$ L (c) = 0 $ if and only if $c\in \bC 1_\cC $. 
\item $\mathrm{Dom}(L):= \{c: L (c) < +\infty\}$ is dense in $\cC $ 
(and so is a dense subspace).
\item Define a metric, $d^L $, on $ S (\cC) $ by
\[
d^L (\mu, \nu) =\sup \{|\mu (c) -\nu (c)|: L (c)\leq 1 \}  .
\]
(A priori this can take the value $+\infty$. Also, 
an argument given just before definition 2.1 of \cite{R6}
shows that it suffices to take the supremum 
only over self-adjoint $c$'s.) We require that the topology 
on $ S (\cC) $ determined by this metric agrees with the 
weak-$*$ topology.
\item $L$ is lower semi-continuous with respect to the 
operator norm.
\end{enumerate}
We will call a pair $(\cC, L)$ with $\cC$ an operator
system and $L$ a Lip-norm on $\cC$ a ``metrized
operator system''.
\end{definition}
The third condition is the one that is often difficult to 
verify for specific examples, but it is crucial for our purposes. 
If the fourth condition is not satisfied, then the closure of $ L$ as defined in \cite{R5} will satisfy it,
with no change in the metric $d^L $. Since a compact space that is metrizable is separable (i.e. has a countable dense subset), any operator system on which a Lip-norm is defined must be separable.

In proposition 1.1 of \cite{R18} it is shown that if $E $ is a countable subset of a separable order-unit space $C $ then there are many Lip-norms on $ C $ that are finite on $E $. Let $ (\cA,\D) $ be a compact quantum group such that $\cA $ is separable, and let 
$\cA_c$ be the corresponding dense coordinate subalgebra. 
It follows that there are many Lip-norms on $\cA$ that are 
finite on $\cA_c $. In \cite{Lih5} Hanfeng Li calls such Lip-norms ``regular''. 
 
 It will be important for us that the Lip-norms that we use are 
 suitably invariant. Here are the definitions given by Li in \cite{Lih5}: 
 
\begin{definition}
 \label{labinv}
A regular Lip-norm $ L $ on a compact 
quantum group $ (\cA,\D) $ is said to be 
 ``left-invariant'' (or ``right-invariant'') if for
all $a\in \cA $ with $a^* =a $ and all $\mu \in S (\cA)$ we have
\[
L (a * \mu)\leq  L (a) \quad (\mathrm{or} \ \ L ( \mu * a)\leq  L (a))  
\]
where $ a* \mu  = (  \mu \otimes I^\cA  )\D(a) $ and similarly
for $\mu * a$.
Then $L $ is said to be ``bi-invariant'' if it is both left and right invariant. 
\end{definition}

Li shows in proposition 8.9  of \cite{Lih5} that if
$ L $ is any regular Lip-norm on $ (\cA,\D) $, and if $ (\cA,\D) $
is coamenable, then one obtains a left-invariant regular Lip-norm
$ L'$ by setting
\[
L'(a) =\sup_{\mu \in S (\cA)} L (a * \mu)   .
\]
There is a similar result producing right-invariant regular 
Lip-norms. Using these constructions, Li shows that one 
can obtain bi-invariant regular Lip-norms. 

We summarize the above results of Li that we need with:
\begin{proposition}
\label{inver}
 Let $ (\cA,\D) $ be a separable coamenable compact 
 quantum group. Then there exist (probably many)
bi-invariant regular Lip-norms on $\cA $.\end{proposition}

We emphasize that the Lip-norms on the algebras
discussed above 
may well not satisfy the Leibniz inequality
\begin{equation}
\label{eqleib}
L(aa') \leq L(a)\|a'\| + \|a\| L(a')  ,
\end{equation}
which they would satisfy if they came from a first-order differential 
calculus or a spectral triple.
Thus they may not relate particularly well to the
products on the algebras. (Of course, for operator systems that 
are not algebras this Leibniz condition has no meaning.) 

\begin{question}
How does one characterize the separable compact 
 quantum groups that admit a bi-invariant regular 
 Lip-norm that satisfies the 
 Leibniz inequality (or, even better, the strong Leibniz condition 
 defined in \cite{R21} and studied in \cite{AGKL}, 
 or better yet, comes from a spectral triple)? 
\end{question}
Examples of ones which do have Lip-norms coming from 
spectral triples can be found in \cite{R18, OzR, ChRi, KKd}. The compact 
quantum groups implicit in the first three of
these papers are those corresponding
to certain classes of finitely generated groups. But many 
related questions remain for other classes 
of finitely generated groups.

In the next section
we will need the results we give below, which are mostly due to Li.
Much as in section 2 of \cite{R5},
we define the radius of a metrized operator
system $(\cC, L) $ to be half of the 
diameter of the compact metric space $ (S (\cC), d^L) $.
(For a thorough discussion of the nuances concerning Lip-norms
on operator systems as opposed to order-unit spaces see
section 2 of \cite{KKd}.)
Equivalently, according to proposition 2.2 of \cite{R5},
$r_\cC $ is the smallest constant, $r$, such that for any $c \in\cC $ 
with $c^* =c $
we have $\|c\|' \leq rL(c) $, where $\|\cdot \|'$ is the quotient 
norm on $\cC/\bC 1_\cC $ (so that there is some $t\in \bR $
such that $\|c-t1_\cC\| \leq  rL(c) $).

\begin{lemma} 
\label{lemrad}
Let $(\cC, L) $ be a metrized operator system.
For any $c\in\cC $ with $c^* =c $ and any $\mu\in S(\cC) $
we have $\|c- \mu(c)1_\cC\| \leq 2r_\cC L(c)$.
\end{lemma}

\begin{proof}
Given $c\in\cC $ with $c^* =c $, choose
 $t\in \bR $
such that $\|c-t1_\cC\| \leq  r_\cC L(c) $).
Then for any $\mu\in S(\cC) $ we have
\begin{align*}
\|c- \mu(c)1_\cC\| &\leq \|c- t1_\cC\| + \|t1_\cC - \mu(c)1_\cC\|  \\
&\leq  r_\cC L(c) + \|\mu(t1_\cC - c)1_\cC\|   
\leq 2r_\cC L(c).
\end{align*} .
\end{proof}

\begin{corollary}
\label{corhr}
(lemma 8.4 of \cite{Lih5})
Let $ (\cA,\D) $ be a compact quantum group with Haar
state $\hr $, 
and let $L $ be a regular Lip-norm on
$ (\cA,\D) $.
Then for any $a\in\cA $ with $a^* =a $ we 
have 
\[
\|a- \hr(a)1_\cA\| \leq 2r_\cA L(a).
\]
\end{corollary}

The following key proposition is a general version of much 
of lemma 8.6 of \cite{Lih5}.

\begin{proposition}
\label{dense}
Let $\cA $ be a unital C*-algebra, and let $ L$ be a 
Lip-norm on $\cA $. Let $ (\cH, \pi) $ be a 
faithful $*$-representation of
$\cA $, and let $\cK $ be a dense subspace of $\cH $.  
Let $S_\cK (\cA) $ be the set of finite convex combinations 
of vector states determined by unit vectors in $\cK $.
Then for every $\mu\in S (\cA) $ and every $\e >0 $
there exists a $\nu\in S_\cK (\cA) $ such that
\[
|\mu (a) -\nu (a)| \leq \e L (a)
\]
for all $a\in\cA $.
\end{proposition}

\begin{proof}
 Let $r_\cA $ be the radius of $ (S(\cA), d^L) $,
as defined before Lemma \ref{lemrad}, and let 
\[
B_ L =\{a\in\cA: \ a^* =a, \ L (a)\leq 1, \ \|a\| \leq r_\cA \}.
\]
Then if  $a^*=a $ and $L (a)\leq 1 $, there is a $t\in\bR $ such that
$a -t1_\cA $ is in $ B_ L $. By the basic criterion for 
verifying property 3 in the definition of a 
Lip-norm (see theorem 4.5 of
\cite{R6}) $B_L $ is a totally bounded subset of $ \cA $.

Let $\mu $ and $\e>0 $ be given. Let $a_1, \cdots, a_m $ be 
elements of $ B_L $ such that 
the balls around them of radius $\e /4 $ cover $ B_ L $. 
It is easily seen that  $S_\cK (\cA) $ is dense in $ S (\cA) $ 
for the weak-$*$ topology. (Argue as for corollary T.5.10 of \cite{WO}.)
Accordingly, we can find a $\nu\in S_\cK (\cA) $ such that
$\|\mu (a_j) -\nu (a_j)\| \leq \e/2 $ for $j =1, \cdots , m $. 

Let $a \in \cA$. We only need to consider 
$a $'s for which $ L (a)<\infty $. By scaling we can assume that
$ L (a) = 1$. For such an $a $ with $a^* =a $,
there is a $t\in\bR $ such that
$a-t1_\cA $ is in $B_ L $. Thus there is a $j_0 $ such that
$\|a -t1_\cA -a_{j_0}\| \leq \e/4 $. Then 
\begin{align*}
|\mu (a) -\nu (a)| &= | (\mu -\nu) (a -t1_\cA)|   \\
&\leq | (\mu -\nu) (a -t1_\cA - a_{j_0})| + | (\mu -\nu) (a_{j_0})|   \\
&\leq \e/2 +\e/2 =\e = \e L (a).
\end{align*}
For general $a \in \cA$ we can replace $\e$ with $\e/2$ in the
above argument, and apply the result to the real and imaginary 
parts of $a$.
\end{proof}

\begin{proposition} (Lemma 8.6 of \cite{Lih5})
\label{profin}
Let $ (\cA,\D) $ be a compact quantum group and let $ L $ be a 
regular Lip-norm on $\cA $. Assume that $ (\cA,\D) $ is in reduced form, 
that is, its Haar state is faithful. Then for any $\mu\in S (\cA) $ and 
any $\e >0 $ there is a $\nu\in S (\cA) $ and a finite subset 
$F \subseteq \hd $ such that 
\[
|\mu (a) -\nu (a)| \leq \e L (a)  
\]
for all $a\in\cA $, but  $\nu (\cA^\g) = 0 $ whenever $\g\in\hd $ but
$\g$ is \emph{not} in $F $.
\end{proposition}

\begin{proof}
Let $ (\cH,\pi) $ be the GNS representation of $\cA $ for the 
Haar state.
Since $ (\cA,\D) $ is in reduced form, this representation is faithful. 
Let $\cK $ be the image in $\cH $ of the dense subalgebra 
$\cA_c $ of $\cA $. Then $\cK $ is a dense subspace of $\cH $.
Let  $\mu\in S (\cA) $ and 
$\e >0 $ be given.
Then according to Proposition \ref{dense} there is a finite convex 
combination, $\nu $, of vector states using vectors from
$\cK $, such that 
\[
|\mu (a) -\nu (a)| \leq \e L (a)
\]
for all $a\in\cA $. Let $a_1, ..., a_m $ be the 
elements of $\cA_c $ of length-one in $\cK$  
determining the vector 
states whose convex combination is $\nu $, 
and for each $j: 1\leq j \leq m $
let $ F_ {a_j} $ be defined as in Proposition \ref{provs}.
Set $ F = \bigcup F_ {a_j} $. From Proposition \ref{provs}
it follows quickly that $F $ has the desired properties.
\end{proof}

%%%%%%%%%%%%%%%%%%%%%% 

\section{Induced Lip-norms for actions}
\label{secgh}

Suppose now that $\a $ is an action of a compact quantum 
group $ (\cA,\D) $ on a unital C*-algebra $ \cB $. If $\cA $ is 
separable then $\hd $ is countable. If also $\a $ is ergodic, 
then $\cB $ is separable. In theorem 1.4 and section 8
of \cite{Lih5} Li 
shows that if $ (\cA,\D) $ also is coamenable (and 
$\a $ is ergodic),
then any regular Lip-norm $ L^\cA$ on $\cA $ induces a 
corresponding Lip-norm 
$ L^\cB $ on $\cB $, defined by
\begin{equation}
\label{eqind}
L^\cB (b) =\sup_{\phi \in S (\cB)} L^\cA(b*\phi)
\end{equation}
where $b*\phi = (\phi \otimes I^\cA) \a(b) $,
and that $ L^\cB $ is regular on $\cB $ in the sense that 
it is finite on $\cB_c $. As Li remarks, this is a generalization 
to the setting of compact quantum groups of the 
corresponding fact for actions of ordinary compact groups 
given in theorem 2.3 of \cite{R4}.
 
 Let $L^\cB$ be any regular Lip-norm on $\cB $. Then Li 
 calls $ L^\cB $ ``$\a$-invariant" if
 \begin{equation}
 \label{eqinv}
  L^\cB ( \mu * b)\leq  L^\cB (b)
 \end{equation}
for all $b \in \cB $ with $b^* = b $ and all $\mu \in S (\cA)$,
where $ \mu * b = (I^\cB \otimes \mu)\a(b)$. Li shows that
if $ L^\cB $ is induced as above from a right-invariant 
regular Lip-norm on $\cA $ then $ L^\cB $ is $\a $-invariant.
(Note that $ (\mu, b) \mapsto \mu* b $ gives in action of
the algebra $\cA'$, with its convolution product, on $\cB $.)

We will now present an important
step in Li's proof of these results, since this step is  
crucial for our purposes. 
\begin{lemma}
\label{lemcid}
Let $ (\cA,\D) $ be a compact quantum group, and let
 $\a $ be an action of $ (\cA,\D) $ on a unital C*-algebra 
 $ \cB $. Assume that $ (\cA,\D) $ is coamenable, and let 
 $\ep $ be the coidentity element (which is a 
 $*$-homomorphism from $\cA $ into $\bC $,
 so an element of $S (\cA) $). Then
 $\ep *b =b $ for all $b \in \cB $, that is,
 $
 (I^\cB \otimes \ep)\circ \a = I^\cB  .
 $
\end{lemma}

\begin{proof}
(a fragment from the proof of lemma 2.2 of \cite{SkZ})
Set $\a_o =  (I^\cB \otimes \ep)\circ \a $. On multiplying 
equation \eqref{eqact} on the left by 
$I^\cB \otimes \ep \otimes I^\cA $ and simplifying, we obtain
\[
(\a_o \otimes I^\cA)\circ \a = \a   .
\]
Then for all $a\in \cA $ and $b \in\cB $ we have
\[
(\a_o \otimes I^\cA) ( (\a (b)) (1_\cB \otimes a))
= ((\a_o \otimes I^\cA)\a (b)) (1_\cB \otimes a)
=  \a (b) (1_\cB \otimes a)  .
\]
Thus $\a_o \otimes I^\cA $ coincides with 
$I^\cB \otimes I^\cA $ on $\a (\cB) (1_\cB \otimes \cA) $. 
But the latter spans a dense subspace of 
$\cB \otimes \cA $ according to the non-degeneracy hypothesis 
in the definition of an action. It follows that
$\a_o \otimes I^\cA  = I^\cB \otimes I^\cA $, and so
$\a_o = I^\cB$ as desired.
\end{proof}

\begin{proposition} (Related to lemma 8.7 of \cite{Lih5})
\label{proav}
Let $ (\cA,\D) $ be a compact quantum group, and let
 $\a $ be an ergodic action of $ (\cA,\D) $ on a unital C*-algebra 
 $ \cB $.
Let $ L^\cA $ be a regular right-invariant Lip-norm on $\cA $. 
Assume that $ (\cA,\D) $ is coamenable, and
let $L^ \cB $ be defined by equation \eqref{eqind},
so that $L^ \cB $ is an $\a$-invariant regular 
Lip-norm on $\cB$.
Let $\ep $ be the coidentity of $\cA$,
viewed as an element of $ S(\cA) $.
Let $\e > 0$ be given, and let $\ep $ be used as $\mu $
in Proposition \ref{profin} to produce the state $\nu $
and the finite set $ F $ with the properties described 
in that proposition. Let $ P_\nu $ be the operator on
$ \cB $ defined by 
\[
P_ \nu (b) =\nu *b = (I^\cB \otimes \nu)\a (b)
\]
for $b \in \cB $.
Then the range of $ P_\nu $ is contained in $\cB^F $
(the direct sum of the isotypic components $\cB^\g $ for
$\g\in F $), and
\[
L^\cB( P_\nu (b)) \leq L^\cB(b) \quad \quad and \quad \quad
\|b -P_\nu (b)\| \leq \e L^\cB (b)   .
\]
for all $b\in\cB $ with $b^* =b $.
\end{proposition}

\begin{proof}
The fact that $L^\cB( P_\nu (b)) \leq L^\cB(b)$ follows 
directly from the definition  of $\a$-invariance \eqref{eqinv}
and the definition of $P_\nu (b) $.
Because $\a (\cB^\g)\subseteq \cB^\g\otimes \cA^\g $
for each $\g\in \hd $, we see from the properties of $\nu $
that the range of 
$ P_\nu $ is contained in $\cB^F $.
Finally, let $b\in\cB $ be given with $b^* =b $.
Notice that for any $\phi \in S(\cB)$, any $\mu \in S(\cA)$,
and any $c \in \cB$ we have
$\phi(\mu * c) = (\phi \otimes \mu)\a(c) = \mu(c * \phi)$.
Then 
\begin{align*}
\|b -P_\nu (b)\| &= \|\ep * (b -P_\nu (b))\|
= \sup_ {\phi \in S (\cB)} |\phi (\ep * (b - \nu * b))|  \\
&=  \sup_ {\phi \in S (\cB)} |\ep (b*\phi ) - \nu (b*\phi )| 
\leq  \sup_ {\phi \in S (\cB)} \e L^\cA(b * \phi)   \\
&=  \e L^\cB(b),
\end{align*}
where we have used Lemma \ref{lemcid} for the first equality,
the self-adjointness of $b$ for the second equality, Proposition
\ref{profin} and the choice of $\nu$ for the inequality, and
equation \eqref{eqind} for the final equality.
\end{proof}

We remark that the $ P_\nu $ above can be viewed as a 
generalization of the $P_n $ used in the proof of theorem
8.2 of \cite {R6}, and that the above Proposition
\ref{proav} can be viewed as a generalization of lemma
8.3 in the proof of that theorem. Also, $ P_\nu $ and
$P_n $ can be viewed as analogues of the classical
Fejer kernels of harmonic analysis.

%%%%%%%%%%%%%%%%%%%%%% 

\section{The main theorem: convergence of truncations}
\label{secconv}

In discussing the convergence of truncations we will use
the quantum 
Gromov-Hausdorff distance that was first introduced in
\cite{R6}. Its setting is order-unit spaces. But the
spaces of operators which we will use below are
operator systems, which are order-unit spaces
with important extra structure. Quite soon after
the appearance of \cite{R6} David Kerr introduced
a stronger version of quantum Gromov-Hausdorff
distance \cite{Krr} that was especially tailored to the setting
of operator systems. It uses spaces of unital 
completely positive
maps into matrix algebras (as generalizations
of the state space),
and Kerr has referred
to it as ``complete Gromov-Hausdorff distance''.
At about the same time Hanfeng Li developed
a fairly different strategy for defining quantum
Gromov-Hausdorff-type distances, and provided
a version tailored for C*-algebras \cite{Lih3},
and a version tailored for order-unit spaces \cite{Lih2}.
Subsequently Kerr and Li wrote a paper \cite{KrL} in 
which they showed that when Li's strategy is applied
to operator systems, it leads to exactly the same
quantum Gromov-Hausdorff-type distance as
Kerr's complete Gromov-Hausdorff distance. 
They then call this ``operator Gromov-Hausdorff
distance''. (For C*-algebras the best current quantum
Gromov-Hausdorff-type distance is \Lat's
dual propinquity \cite{Ltr4}, but since it explicitly
uses the Leibniz inequality, it can not be
applied to operator systems.)  

Since most of the spaces of operators used below are
operator systems, it would be appropriate to
use here Kerr and Li's operator Gromov-Hausdorff distance.
But I have chosen not to do this since it would
considerably complicate the notation, and so
somewhat obscure the ideas. But I  fully
expect that with small adjustments the
arguments given below would work well for
operator Gromov-Hausdorff distance, though
I have not checked thoroughly that this is the
case.

For the readers' convenience we begin by 
recalling definition 4.2 of \cite{R6}, which is 
the definition of quantum 
Gromov-Hausdorff distance, adapted here for operator systems
much as in \cite{KKd}.
Let $ (\cC,  L^\cC ) $ and $ (\cD,  L^\cD)  $ be metrized operator systems.
Let $\cM ( L^\cC,  L^\cD) $
be the set of all Lip-norms $L$ on $\cC \oplus \cD $ such 
that $\mathrm{Dom}(L) = \mathrm{Dom}(L^\cC) \oplus \mathrm{Dom}(L^\cD)$ 
and such that the quotient of $ L $ on $\cC $ coincides with  $L^\cC $
and similarly for $\cD $. For this condition the inclusion of
$ S (\cC) $ into $ S (\cC \oplus \cD) $ is an isometry
for the metric $d^ {L^\cC} $ on $ S (\cC) $ and the metric
$d^ L $ on  $ S (\cC \oplus \cD) $, and similarly for $ \cD $,
so we can view  $ S (\cC) $ and  $ S (\cD) $ as subspaces
of  $ S (\cC \oplus \cD) $ with the induced metric from $d^L$.
Then the quantum Gromov-Hausdorff distance,
$\mathrm{dist}_q (\cC,\cD) $, between $\cC $ and $\cD $
is defined by
\[
\mathrm{dist}_q (\cC,\cD) 
= \inf\{\mathrm{dist}_H^{d^L} (S(\cC),S (\cD)):
L\in \cM ( L^\cC,  L^\cD) \}   ,
\]
where $\mathrm{dist}_H $ denotes ordinary
Hausdorff distance. 

Thus for any particular example, the challenge is to construct 
elements $L$ of $\cM ( L^\cC,  L^\cD) $ that bring
$S(\cC) $ and $S (\cD) $ appropriately close together.
A convenient way to approach this (section 5 of \cite{R6}) 
is to look for $L $'s of the form
\[
L (c,d) =  L^\cC (c) \vee  L^\cD (d) \vee  N (c,d)
\]
for $c \in \cC $ and $d\in\cD $ (and $\vee $ means ``max''). 
Here $N $ should be a 
norm-continuous seminorm on $\cC \oplus \cD $ such that
$ N (1_\cC,1_\cD) =0 $ but $ N (1_\cC, 0_\cD)\neq 0 $, 
and such that for any $c \in \cC $ and $\e>0 $ there is a $d\in\cD $
such that 
\[
 L^\cD (d)\vee N (c,d) \leq L^\cC (c) +\e , 
\]
and similarly for $\cC $ and $\cD $ interchanged.
 In \cite{R6}
the term $N$ is called a ``bridge''.

We now assume that $ (\cA,\D) $ is a compact matrix quantum group, 
and we let $(\cH, u) $ be a 
fundamental unitary corepresentation
of $ (\cA,\D) $, as defined in Definition \ref{deffnd}. As discussed 
just before that definition, arrange that $(\cH, u) $ is 
self-conjugate and contains the 
trivial corepresentation. Let $S $ be the set of irreducible 
unitary corepresentations that appear in the decomposition of
$(\cH, u) $. For each $n \in\bN$ we define
$S^n $ as done just before Proposition \ref{profil}, so the
$S^n $'s form a ``filtration'' of $\hd $. 

We also assume that $\a $ is an ergodic action of 
$ (\cA,\D) $ on a unital C*-algebra $\cB$, and we let the
$\cB^{S^n}$'s be defined just as before Proposition \ref{profil}, 
so that they form a filtration of $\cB $. 

The following theorem,
which is the main theorem of this paper, can be viewed
as a generalization of theorem 8.2 of \cite{R6} (which is the 
case in which our compact matrix quantum group 
is an ordinary compact Lie group). 

\begin{theorem}
\label{main}
Let $ (\cA,\D) $ be a coamenable compact matrix quantum group, and
let $\a $ be an ergodic action of 
$ (\cA,\D) $ on a unital C*-algebra $\cB$.
Let notation be as above. Let $ L^\cA $ be a regular Lip-norm on
$\cA $ which is right invariant, and let 
$ L^\cB $ be the seminorm on $\cB $ 
defined by Equation \ref{eqind} (which is a regular $\a$-invariant
Lip-norm). For each $n\in\bN $ let 
$ L^n $ be the restriction of $ L^\cB $ to the operator system
$\cB^{S^n}$, so that $ ( \cB^{S^n} , L^n) $ is a metrized
operator system. Then
\[
\mathrm{dist}_q (\cB^{S^n},\cB) \ \to \ 0  
\quad \mathrm{as} \quad n \ \to \ \infty  . 
\]
\end{theorem}

\begin{proof}
Let $\e>0 $ be given. According to Proposition \ref{proav}
we can find a state $\nu $ of $\cA $ and a finite subset 
$F $ of $\hd $ such that the range of $ P_\nu $ is 
contained in $\cB^F $ and
\[
L^\cB( P_\nu (b)) \leq L^\cB(b)\quad \quad and \quad \quad
\|b -P_\nu (b)\| \leq \e L^\cB (b)   
\]
for all $b \in \cB$ with $b^* = b$.
Since $ F $ is finite, we can find an $N\in \bN $ such that
$ F\subseteq S^N $. Then for all $n\geq N $ we have
$ F\subseteq S^n $ so that the range of $ P_\nu $ is 
contained in $\cB^{S^n} $. 
We can then immediately
apply proposition 8.5 of \cite{R6} to conclude that
$\mathrm{dist}_q (\cB^{S^n},\cB) < \e $ for all $n\geq N $.
The bridge for this situation is simply $\e^{-1}\|b-a\|$.
\end{proof}

For the convenience of the reader, we now state
proposition 8.5 of \cite{R6}, for the case of 
metrized operator systems. We let $\cA^{sa}$
denote the set of self-adjoint elements of $\cA$,
and similarly for $\cB$.

\begin{proposition} (proposition 8.5 of \cite{R6})
Let $ (\cA, L^\cA) $ be a metrized operator system, 
and let $\cB $ be an operator
subsystem of $\cA $. Let $ L^\cB $ denote the restriction of 
$ L^\cA $ to $\cB $, so that $ (\cB, L^\cB )$ is a metrized 
operator system.  Let $ P $ be a function (not 
necessarily even linear or continuous) from $\cA^{sa} $ to 
$\cB^{sa} $ for which there is an $\e >0 $
such that
\[
L^\cB (P (a)) \leq L^\cA(a) \quad \quad and \quad \quad
\|a -P(a)\| \leq \e L^\cA (a)   
\]
for all $a \in \cA^{sa}$.
Then $\mathrm{dist}_q (\cB,\cA) < \e$.
\end{proposition}

%%%%%%%%%%%%%%%%%%%%%% 

\section{Examples}
\label{secex}

We will now give a number of examples to which our results above apply.

\begin{example}
Let $G $ be a compact Lie group, and let $\cA = C(G)$. 
Choose an $Ad$-invariant inner product on the 
Lie algebra of $G $, and let $D $ be the corresponding Dirac operator 
on the Hilbert
space $\cS$ of spinor fields, as described in many places, for 
example in \cite{R35, R22}.
Then
$ (\cA,\cS,D) $ is a spectral triple, and so one 
can define a seminorm, $L ^D$, (with value $+\infty$ allowed) on $ \cA $ by
\[
 L^D (a) =\| [D,a]\|
\]
for any $a\in\cA $. Then $L^D $ is a C*-metric, and in particular
$ L^D$ is a Lip-norm that satisfies the 
Leibniz inequality \ref{eqleib}. See proposition 6.5 of \cite{R35},
as well as its anticedent theorem 4.2 of \cite{R4}.
Our results in the preceding sections apply to this class of examples,
including to the many ergodic actions of compact Lie groups 
on unital C*-algebras \cite{R35}.

The case in which $ G $ is the 
circle group is the example treated in section 3.2 of \cite{vSj2} concerning
Fej\'er-Riesz operator systems.

If we only have a continuous length function on $G$, it too can be used
to define a Lip-norm satisfying the Leibniz inequality on any unital
C*-algebra on which $G$ has an ergodic action. When our results
of previous sections are applied, one obtains theorem 8.2
of \cite{R6}.
\end{example}

\begin{example}
\label{exgp}
Let $\G $ be a finitely generated group, as in Example \ref{dg}. 
Both its full and its reduced C*-algebras are compact quantum 
groups. We can view them as (co)acting on themselves on the 
left using $\D$. 
They both acquire a filtration consisting of operator systems from any 
given finite set $S$ of generators of $\G$ closed under
taking inverses and containing the identity element of $G$. 
Let $\cA = C^*_r(\G) $, the reduced C*-algebra,
with its faithful representation on $\cH =\ell^2 (\G) $, and 
let $\{\cA_n\} $ be the 
corresponding filtration by operator systems using $S$. 
For each $n $ let $\cH_n $ be the image 
of $\cA_n $ in $\cH $, so that the $\cH_n $'s form an increasing
family of finite dimensional subspaces whose union is dense, with $\cH_0$ the
span of $1_\cA$. Set $\cK_0 = \cH_0$, and
for each integer $n\geq 1$ set $\cK_n = \cH_n \ominus \cH_{n-1}$.
Let $ D$ be the unbounded operator on 
$\cH $ whose domain is the algebraic sum $\oplus \cK_n $, 
and which multiplies all elements of $\cK_n $ by $n $, for
all $n \in \bN$. Then
$ (\cA,\cH,D) $ is a spectral triple. This is the main class of
examples discussed in Connes' first paper \cite{Cn7} on the
metric aspects of non-commutative geometry. One can 
again define a seminorm, $L $, on $ \cA $ by
\[
 L (a) =\| [D,a]\|
\]
for any $a\in\cA $. It is easily seen that $ L$ satisfies properties 1,
2 and 4 of the definition \ref{deflip} of a Lip-norm, as well as the
Leibniz inequality \ref{eqleib}. But property
3 of definition \ref{deflip} is only known to hold 
in the case of finitely generated
groups of polynomial growth \cite{R18, ChRi} (so virtually 
nilpotent), and the case of hyperbolic groups
and some related free product groups \cite{OzR}. No counterexamples 
are known for other groups. It is a very interesting open question 
to determine for which other groups property 3 is satisfied. (The
corresponding spectral triples for group algebras twisted
by a 2-cocycle are studied in \cite{LnW21}, but that is 
somewhat outside the the scope of the present paper.)

Of course one can always use Li's results discussed above
to choose a regular Lip-norm $ L$ on $\cA $ that is invariant for the 
left (co)action of $\cA $ on itself (but which may well not satisfy the
Leibniz inequality). But it is only if $\cA $ is coamenable, that 
is, if $\G $ is amenable, that we can apply the results in the earlier sections
to conclude that the operator systems $\cA_n $, equipped with the 
restrictions of $L $ to them, converge to $ (\cA, L) $ for quantum 
Gromov-Hausdorff distance.
\end{example}

\begin{example}
\label{exrapid}
Let $G$ be a group of rapid decay. (See \cite{AnC} for the definition.) 
In  \cite{AnC} it is shown how to use a proper length function
on $G$ and suitable Sobolev-type norms to define in a natural
way Lip-norms on $C^*(G)$. Then when $G$ is amenable 
it is easily seen that our
results in earlier sections apply. But these Lip-norms 
will seldom satisfy the
Leibniz inequality. (The
corresponding Lip-norms for the group algebra of $G$ twisted
by a 2-cocycle are studied in \cite{LnW17}, but that is again 
somewhat outside the the scope of the present paper.)
\end{example}

\begin{example}
\label{exneg}
In \cite{JnM} M. Junge and T. Mei use the theory of one-parameter
semigroups of completely positive operators to show how to
use a conditionally negative functions on a group $G$ of rapid
decay to produce a Lip-norm on
the reduced C*-algebra of $G$ that is Leibniz, and even
strongly Leibniz in the sense defined in definition 1.1 
of \cite{R21} and studied in \cite{AGKL}.
They show that this applies, for example, to cocompact
lattices in certain semisimple Lie groups. There is no
suggestion that there is a Dirac-type operator associated to this
situation. But in \cite{JMP} several Dirac-type operators are examined
that are associated to this situation (starting two paragraphs before proposition
C.4). In chapter 5 of the book \cite{ArK}, especially in section 5.8,
there is further examination of such Dirac-type operators. 
(I thank C\'edric Arhancet, co-author of this book, for bringing this
book and its chapter 5, and thus also reference \cite{JMP}, to my attention.)
\end{example}

\begin{example}
\label{exmatr}
Let $(\cA, \D) $ be any coamenable compact matrix quantum group.
So it is finitely generated, and any finite set of generators 
will yield a filtration of it. 
View it as (co)acting on itself on the 
left using $\D$. 
Because it is coamenable, 
our results above apply to it. Thus we can use Li's results 
discussed above
to choose a regular Lip-norm $ L$ on $\cA $
that is invariant for the left (co)action of $\cA $ on itself. 
Then when the operator systems of the filtration are
 equipped with the restriction of $L $ to them, they converge to 
 $ (\cA, L)$ for quantum 
Gromov-Hausdorff distance.
\end{example}

\begin{example}
\label{exqrap}
In \cite{BhVZ} the authors give a definition of what it means for a
discrete quantum group to have rapid decay (by modifying a definition 
given earlier by Vergnioux for the unimodular case). They then generalize 
the results from \cite{AnC} described in Example
\ref{exrapid} by showing that for any compact quantum matrix
group $(\cA, \D)$ whose dual discrete quantum group has rapid decay
one can again use suitable Sobolev-type norms to define in a natural
way Lip-norms on $(\cA, \D)$. When  $(\cA, \D)$ is coamenable 
it is easily seen that our
results in earlier sections apply. But again,
these Lip-norms will seldom satisfy the
Leibniz inequality.
\end{example}

\begin{example}
\label{exsu}
Let $\cA $ be the compact quantum group $SU_q (2) $. 
For its definition and properties see \cite{KKd} and the many 
references contained therein. 
It is a compact matrix quantum group, and is coamenable.
 
Even better, in \cite{KKd} the authors construct for each $q$ (for
$0<q\leq 1$) a 1-parameter 
family of Dirac-type
operators $D_{t,q}$ on $SU_q (2) $, each of which they prove
determines a regular
Lip-norm on $SU_q (2) $. These Lip-norms satisfy a twisted
Leibniz inequality (lemma 4.8 of \cite{KKd}). And from our results 
described above, for each of these
Lip-norms the operator systems of the filtration coming from 
any faithful finite-dimensional unitary corepresentation of  $ SU_q (2) $ 
(modified as discussed above so as to determine an operator 
subsystem of $SU_q (2) $) 
will converge to 
 $ SU_q (2) $ for quantum 
Gromov-Hausdorff distance.
\end{example}

\begin{example}
\label{expod}
Let $\cB $ be the standard Podle\'s sphere, $C(S^2_q)$. 
For its definition and properties see 
\cite{AgK, AgKK, KKd} and the many 
references contained therein. We can view $C(S^2_q)$ 
as a subalgebra of $SU_q (2) $, and as such, as an embedded 
homogeneous space in $SU_q (2) $. Here, for a compact 
quantum group $(\cA, \D) $ we say \cite{DCm}
that a unital C*-subalgebra
$\cB $ of $\cA $ is an embedded homogeneous space
of $(\cA, \D) $ if $\D (\cB) \subseteq \cB \otimes \cA $ 
so that the restriction of $\D $ to $\cB $ is an action of 
$\cA $ on $\cB $.

There have been many proposals for Dirac operators on $C(S^2_q)$.
Let us denote one of these proposals, that of 
Dabrowski and Sitarz \cite{DbS}, by $D_q $. 
It is shown in \cite{AgK} that the corresponding seminorm, 
$L^{D_q} $,  
is in fact a Lip-norm, and that $L^{D_q} $ is $\a$-invariant,
where here the action $\a$ is just the restriction of $\D$ to the
subalgebra $C(S^2_q)$.  
For the Dirac operators $D_{t,q}$ on 
$SU_q (2) $ of the previous example, let $L^{D_{t,q}} $
be the corresponding Lip-norms. In proposition 5.2 of
\cite{KKd} it is shown that for each $t $ the restriction of
$L^{D_{t,q}} $ to $C(S^2_q)$ is $L^{D_q} $. To put this in the 
context of Li's framework, we use the following simple result.

\begin{proposition}
Let $(\cA, \D) $ be a coamenable compact quantum group,
and let $\cB $
be an embedded homogeneous space in $\cA $. Let $ L^\cA$
be a regular Lip-norm on $\cA $. Let $ L^\cB $ be Li's 
corresponding induced Lip-norm on $\cB $ as defined in
equation \eqref{eqind}. Then $ L^\cB $ coincides 
with the restrictions of $ L^\cA$ to $\cB $.
\end{proposition}
\begin{proof}
Let $\mu \in S (\cA)$, and let $\phi $ be its restriction to $\cB $,
so $\phi \in S (\cB) $.
Since for any $b \in \cB $ we have $\D (b)\in \cB\otimes \cA $,
we see that
\[
b*\phi = (\phi \otimes I^\cA)\D (b) =  (\mu \otimes I^\cA)\D (b)
=b *\mu    .
\]
But any $\phi \in S (\cB) $ can be extended (perhaps in many ways) 
to be an element $\mu $ of $ S (\cA) $. It follows easily that
$ L^\cB (b) \leq  L^\cA (b) $. But if we let $\mu =\ep $
we see that in fact $ L^\cB (b) =  L^\cA (b) $.
\end{proof}

Consequently, the filtration of $C(S^2_q)$  constructed as in
Proposition \ref{profil} from any faithful finite-dimensional unitary 
corepresentation of  $ SU_q (2) $ (modified as discussed above 
so as to determine an operator subsystem of $SU_q (2) $) 
will converge to 
$C(S^2_q)$ for quantum 
Gromov-Hausdorff distance.
\end{example}

\begin{example}
\label{exberez}
The situation involving ``matrix algebras converge to the sphere''
which is discussed in \cite{R7, R21, R29}
can be viewed as being a closely related situation in which 
extra structure is present that equips the finite-dimensional
operator systems with a C*-algebra product making them full 
matrix algebras. But when this is generalized to the Podle\'s spheres
\cite{Sai, AgKK} one gets only operator systems.
\end{example}

%%%%%%%%%%%%%%%%%%

\providecommand{\bysame}{\leavevmode\hbox to3em{\hrulefill}\thinspace}
\providecommand{\MR}{\relax\ifhmode\unskip\space\fi MR }
\providecommand{\MRhref}[2]{%
  \href{http://www.ams.org/mathscinet-getitem?mr=#1}{#2}
}
\providecommand{\href}[2]{#2}

}   % end Large

\end{document}